%% file: ms.tex
\documentclass{article}

\usepackage{arxiv}

\usepackage[utf8]{inputenc} 
\usepackage[T1]{fontenc}    
\usepackage{hyperref}       
\usepackage{url}            
\usepackage{booktabs}       
\usepackage{amsfonts}       
\usepackage{nicefrac}       
\usepackage{microtype}      

\usepackage{amsmath, amssymb, amsthm}
\usepackage{algorithm, algorithmic}

\usepackage{mathtools}
\DeclarePairedDelimiter\ceil{\lceil}{\rceil}
\DeclarePairedDelimiter\floor{\lfloor}{\rfloor}

\usepackage{caption}
\usepackage{subcaption}
\usepackage{color}
\usepackage{arydshln}

\newtheorem{Theorem}{Theorem}[section]
\newtheorem{Property}{Property}[section]

\newtheorem{Assumption}{Assumption}[section]

\usepackage{tikz}
\usetikzlibrary{graphs,graphs.standard,quotes}

\title{Rainbow Perfect and Near-Perfect Matchings in \\
Complete Graphs with Edges Colored by \\
Circular Distance}

\author{
    Shuhei~Saito\\
    Seikei University\\
    \And
    Wei Wu\thanks{Corresponding author. Tel: +81-53-478-1213; fax: +81-53-478-1213; email: \texttt{goi@shizuoka.ac.jp}.}\\
    Shizuoka University\\
    \And
    Naoki~Matsumoto\\
    Keio University\\
}

\begin{document}
\maketitle

\begin{abstract}
Given an edge-colored complete graph $K_n$ on $n$ vertices, a perfect (respectively, near-perfect) matching $M$ in $K_n$ with an even (respectively, odd) number of vertices is rainbow if all edges have distinct colors.
In this paper, we consider an edge coloring of $K_n$ by circular distance,
and we denote the resulting complete graph by $K^{\bullet}_n$.
We show that when $K^{\bullet}_n$ has an even number of vertices, it contains a rainbow perfect matching if and only if $n=8k$ or $n=8k+2$, where $k$ is a nonnegative integer.
In the case of an odd number of vertices, Kirkman matching is known to be a rainbow near-perfect matching in $K^{\bullet}_n$.
However, real-world applications sometimes require multiple rainbow near-perfect matchings.
We propose a method for using a recursive algorithm to generate multiple rainbow near-perfect matchings in $K^{\bullet}_n$.
\end{abstract}


\section{Introduction}
Given an \mbox{edge-colored} undirected graph $G$, a \emph{rainbow matching} (or \emph{heterochromatic matching}) $M$ in $G$ is a matching (a set of edges without common vertices) such that all edges have distinct colors.
While it is possible to find a maximum matching in $G$ in polynomial time,
computing a maximum rainbow matching is known to be an \mbox{NP-hard} problem.
Indeed, the decision version of this problem is a classical example of \mbox{NP-complete} problems, even for edge-colored bipartite graphs~\cite{garey1979computers}.

If a graph $G$ has an even number of vertices,
a \emph{rainbow perfect matching} in $G$ is a rainbow matching that matches all vertices of the graph. If the number of vertices is odd, a \emph{rainbow near-perfect matching} $M$ in $G$ is a rainbow matching in which exactly one vertex is unmatched.
Note that in this paper, we use \emph{RPM} to mean either a rainbow perfect matching or a rainbow near-perfect matching in the corresponding graph.

Over the past decade, finding rainbow perfect and near-perfect matchings in edge-colored graphs or hypergraphs has been studied for several classes of graphs,
including
complete bipartite graphs~\cite{perarnau2013rainbow}, $r$-partite graphs~\cite{cano2016rainbow},
Dirac bipartite graphs~\cite{coulson2019rainbow},
random geometric graphs~\cite{bal2017rainbow},
and $k$-uniform, $k$-partite hypergraphs~\cite{bal2016rainbow}.
Most of the above studies assume that there are more colors used for coloring than there are colors among matchings.
If we do not insist on perfect matchings, other studies have demonstrated the existence of large rainbow matchings in arbitrarily edge-colored graphs.
Letting $\hat{\delta}(G)$ denote the minimum color degree 
of an edge-colored graph $G$,
Wang and Li~\cite{wang2008heterochromatic} first showed the existence of a rainbow matching whose size is $\ceil*{\frac{5\hat{\delta}(G)-3}{12}}$,
and conjectured that a tighter lower bound $\ceil*{\frac{\hat{\delta}(G)}{2}}$ exists if $\hat{\delta}(G)\ge 4$ holds.
LeSaulnier~et~al.~\cite{lesaulnier2010rainbow} proved a weaker statement, that
an edge-colored graph $G$ always contains a rainbow matching of size at least $\floor*{\frac{\hat{\delta}(G)}{2}}$.
Kostochka and Yancey~\cite{kostochka2012large} completed a proof of Wang and Li’s conjecture.
Letting $n$ be the size of vertices of a graph $G$, Lo~\cite{lo2015existences} showed that an edge-colored graph $G$ contains a rainbow matching of size at least $k$, where $k=\min\left\{\hat{\delta}(G), \frac{2n-4}{7}\right\}$.
If the graph is properly edge-colored
(i.e., no two adjacent edges have the same color),
several results~\cite{wang2011rainbow,diemunsch2011rainbow,gyarfas2014rainbow,aharoni2019large} have provided lower bounds for the size of a maximum rainbow matching.
Other recent studies~\cite{babu2015rainbow,wang2016existence,cheng2018note} have focused on finding large rainbow matchings under the stronger assumption that the given is strongly edge-colored.
See Kano and Li~\cite{kano2008monochromatic} for a deeper analysis of rainbow subgraphs in an edge-colored graph.

Many results related to rainbow matchings in complete graphs have also been obtained from Ramsey-type problems.
Letting $M_k$ denote a matching of size $k$,
Fujita~et~al.~\cite{fujita2009rainbow} determined $AR(M_k,n)$ for $k\ge 2$ and $n\ge 2k+1$ and provided a conjecture regarding $AR(M_k,2k)$,
where the anti-Ramsey value $AR(H,n)$ is the smallest integer $r$ such that for any exact 
\mbox{$r$-edge} coloring of $K_n$, there exists a subgraph isomorphic to $H$ that is rainbow.
Haas and Young~\cite{haas2012anti} confirmed the conjecture regarding $AR(M_k,2k)$ for $k\ge 3$.
More Ramsey-type results can be found in a survey by Fujita, Magnant, and Ozeki~\cite{fujita2010rainbow}.
To the best of our knowledge, there has been no research on finding RPMs in edge-colored complete graphs with exactly $\floor*{\frac{n}{2}}$ colors.
If we consider any exact 
\mbox{$\floor*{\frac{n}{2}}$-edge} coloring of $K_n$ where $n\ge 4$,
the results in~\cite{fujita2009rainbow} show that the size of a maximum rainbow matching is bounded by $\mathcal{O}\left(\sqrt{\frac{n}{2}}\right)$.
This indicates that there does not always exist an RPM in $K_n$ by an arbitrary \mbox{$\floor*{\frac{n}{2}}$-edge} coloring.
Thus, we consider RPMs in edge-colored complete graphs by a special \mbox{$\floor*{\frac{n}{2}}$-edge} coloring, called a circular-distance edge coloring.
We show results for the existence and non-existence of such RPMs,
and then propose a method for using a recursive algorithm to generate multiple RPMs when $n$ is odd.

In this paper, all graphs are simple and undirected.
A graph $G$ is defined as $G=(V,E)$, where $V$ (or $V(G)$) is the set of vertices and $E$ (or $E(G)$) is the set of edges.
We define a vertex set $V$ containing $n$ vertices as $V= \{0, 1, \ldots, n-1\}$, unless indicated otherwise.
Given a graph $G$ and a set of distinct colors $C=\{c_1,c_2,\ldots\}$,
the \emph{\mbox{circular-distance} edge coloring} of a graph $G$ is a mapping $h: E(G)\rightarrow C$, where an edge $\{i,j\}\in E(G)$ is colored as
\begin{align}\label{def:coloring}
    h(\{i,j\}) = c_{\min\{|i-j|, n-|i-j|\}}.
\end{align}
We denote by $K^\bullet_n$ the \mbox{edge-colored} complete graph $K_n$ with each edge $e$ colored by $h(e)$.
According to coloring~\eqref{def:coloring}, $K^\bullet_n$ is colored in exactly $\floor*{\frac{n}{2}}$ colors, and it is not properly colored if $n\ge 3$.
Here, we aim to find an RPM $M$ in $K^\bullet_n$, that is, a rainbow matching $M$ where $|M| = \floor*{\frac{n}{2}}$.
Figures~\ref{fig:k*7} and \ref{fig:k*8} respectively show examples of $K^\bullet_7$ and $K^\bullet_8$ with their RPMs.
\input{K7-K8}

Given an RPM $M$ in $K^\bullet_n$,
an \emph{$\alpha$-rotated matching} of $M$ denoted by $\mathrm {rot}(M, \alpha)$ is defined as
\begin{align}\label{eq:m-rot}
\mathrm {rot}(M, \alpha) = \{\{(i+\alpha) \bmod n,(j+\alpha) \bmod n\} \mid \forall \{i,j\}\in M \}.
\end{align}
Figure~\ref{fig:rpm-rk7} shows the \mbox{1-rotated} matching $\mathrm {rot}(M,1)$ for the RPM $M$ in Figure~\ref{fig:rpm-k*7}.
\input{RotateK7}
This example shows that if we arrange vertices $0, 1, \ldots, n-1$ around a cycle in clockwise order, $\mathrm {rot}(M,\alpha)$ matches by rotating matching $M$ by $\frac{2\alpha}{n}\pi$ in the clockwise direction.

If $M$ is an RPM in the graph $K^\bullet_n$, we observe that the following properties hold from definition~\eqref{eq:m-rot}:
\begin{Property}\label{pro:rot}
$\forall \alpha \in \mathbb{Z}, \mathrm{rot}(M,\alpha)$ is an RPM,
\end{Property}
\begin{Property}
$M = \mathrm{rot}(M,0)$,
\end{Property}
\begin{Property}
If $\alpha' \equiv \alpha'' \pmod{n}$, $\mathrm{rot}(M,\alpha') = \mathrm{rot}(M,\alpha'')$, 
\end{Property}

Next, we define the reversed matching of $M$ as $\mathrm{rev}(M)$:
\begin{align}\label{eq:m-rev}
\mathrm{rev}(M) = \{\{n-1-i,n-1-j\} \mid \forall \{i,j\}\in M\}.
\end{align}
Figure~\ref{fig:rpm-rk8} shows the reversed matching $\mathrm{rev}(M)$ for the RPM $M$ in Figure~\ref{fig:rpm-k*8}.
\input{ReverseK8}
Given an RPM $M$ in the graph $K^\bullet_n$, the following properties hold by definition~\eqref{eq:m-rev}:
\begin{Property}\label{pro:rev}
The matching $\mathrm{rev}(M)$ is also an RPM in the graph $K^\bullet_n$,
\end{Property}
\begin{Property}
$\mathrm{rev}(\mathrm{rev}(M)) = M$.
\end{Property}

\subsection{\mbox{Round-robin} tournament scheduling problem}
In real-world applications, some combinatorial problems and their sub-problems are related to searches for an RPM in $K^\bullet_n$.
A well-known example is scheduling for round-robin tournaments.

Given $n$ teams (where $n$ is even), a tournament organizer must decide which games take place in which rounds.
The \emph{round-robin tournament scheduling problem} (RTSP) aims to generate a schedule with $n-1$ rounds such that
\begin{itemize}
    \item each pair of teams is matched exactly once, and
    \item each team plays exactly one game in each round.
\end{itemize}

Translating this problem into the language of graph theory, let each team be a vertex, and let an edge connecting vertices $i$ and $j$ represent a game between teams $i$ and $j$.
A perfect matching in the complete graph $K_n$ thus describes a round of $n/2$ games.
Therefore, the RTSP with $n$ teams is the problem of decomposing the complete graph $K_n$ into $n-1$ perfect matchings~\cite{de1980geography, de1981scheduling,januario2016edge}.

Next, we show that a decomposition of $K_n$ can be formed based on any rainbow \mbox{near-perfect} matching $M$ in $K^\bullet_{n-1}$.
Letting vertex $i' \in V(K^\bullet_{n-1})$ be a vertex not matched by $M$,
\begin{align}
\label{eq:decomp}   &\mathrm{rot}(M,i)\cup\{\{(i'+i) \bmod{(n-1)}, n-1\}\}, &\forall i \in \{0, 1, \ldots, n-2\}
\end{align}
is a feasible decomposition of $K_n$.
Another decomposition using the reversed RPM $\mathrm{rev}(M)$ can be similarly constructed as
\begin{align}
\label{eq:decomp-rev}   &\mathrm{rot}(\mathrm{rev}(M),i)\cup\{\{(n-1-i'+i) \bmod{(n-1)}, n-1\}\}, &\forall i \in \{0, 1, \ldots, n-2\}.
\end{align}

Figure~\ref{fig:rtsp8} shows a feasible schedule for $n=8$ teams (a decomposition of $K_8$) using method~\eqref{eq:decomp} based on the RPM in Figure~\ref{fig:rpm-k*7}.
\input{RTSP8}
The matchings, generated by method~\eqref{eq:decomp} or method~\eqref{eq:decomp-rev}, partition the edge set of $K_8$ into $7$ perfect matchings.
Kirkman~\cite{Kirkman1847on} first proposed this framework for scheduling round-robin tournaments.
This approach uses a \emph{Kirkman matching} $M^\mathrm{kir}_n$ in $K^\bullet_n$ as
\begin{align}\label{eq:Kirkman-matching}
    M^\mathrm{kir}_n = \left\{\{i,n-1-i\} \mathrel{\Big|} \forall i \in \left\{0,1,\ldots, \floor*{\frac{n}{2}}-1\right\}\right\}.
\end{align}
We call a schedule generated from Kirkman matching a \emph{Kirkman schedule}.
Figure~\ref{fig:rtsp8} shows a Kirkman schedule with 8 teams.
A Kirkman matching $M^\mathrm{kir}_n$ has the following special properties:
\begin{Property}\label{pro:kick-RPM}
A Kirkman matching $M^\mathrm{kir}_n$ in $K^\bullet_n$ is an RPM if and only if $n=2$ or $n$ is odd.
\end{Property}
\begin{Property}\label{pro:kick-rev}
For an RPM $M$ in the graph $K^\bullet_n$,
if there exists $\alpha$ such that  $\mathrm{rot}(M,\alpha)=\mathrm{rev}(M)$, then there exists $\beta$ such that $M=\mathrm{rot}(M^\mathrm{kir}_n,\beta)$.
\end{Property}
Property~\ref{pro:kick-rev} indicates that Kirkman matching $M^\mathrm{kir}_n$ and its $\alpha$-rotated matchings are the only matchings that make decomposition~\eqref{eq:decomp} and decomposition~\eqref{eq:decomp-rev} the same.

\subsection{Normalization}
Although Kirkman scheduling is popular for generating schedules for round-robin tournaments in European soccer leagues~\cite{goossens2012soccer},
it can produce unbalanced or unfair schedules~\cite{miyashiro2006minimizing,lambrechts2018round}.
Other feasible solutions to the RTSP (non-Kirkman schedules) may thus be needed for the next scheduling stage.
To obtain different solutions by using method~\eqref{eq:decomp} and method~\eqref{eq:decomp-rev},
we first design an algorithm $\textsf{norm}(M)$ that normalizes an RPM $M$ using reverse and rotate operations.
Algorithm~\ref{alg:norm} shows details of this normalization.
We first clockwise rotate the matching $M$ until edge $\{0,n-1\}\in M$ meets (lines~\ref{line:norm-rot-start}--\ref{line:norm-rot-end}).
We then check each edge $\{i,j\}$ in $M$ in ascending order of their color indexes.
For the current edge $\{i,j\}$, if $i+j>n-1$, the normalization finishes and returns $\mathrm{rev}(M)$;
if $i+j<n-1$, the normalization ends with the current $M$ (lines~\ref{line:norm-rev-start}--\ref{line:norm-rev-end}).
The normalization procedure returns the current $M$ when all edges have been checked (line \ref{line:norm-end}).

\begin{algorithm} 
\caption{Normalize a rainbow perfect matching: $\textsf{norm}(M)$}
\label{alg:norm}
\begin{algorithmic}[1]
\REQUIRE{an RPM $M$.}
\STATE Let $\{i,j\}$ be the edge colored in $c_1$ in $M$ and $j>i$.\label{line:norm-rot-start}
\IF{$\{i,j\} \ne \{0,n-1\}$}
    \STATE $M \leftarrow \mathrm{rot}(M,-j)$.
\ENDIF\label{line:norm-rot-end}
\FOR{$k=2$ \TO $k=\floor*{\frac{n}{2}}$}\label{line:norm-rev-start}
    \STATE Let $\{i,j\}$ be the edge colored in $c_k$ in $M$.
    \IF{$i+j < n-1$}
        \STATE \textbf{return} $M$.
    \ELSIF{$i+j > n-1$} 
        \STATE \textbf{return} $\mathrm{rev}(M)$.
    \ENDIF
\ENDFOR\label{line:norm-rev-end}
\STATE \textbf{return} $M$.\label{line:norm-end}
\end{algorithmic}
\end{algorithm}

Properties~\ref{pro:rot} and \ref{pro:rev} ensure that if $M$ is an RPM in the graph $K^\bullet_n$, the normalization $\textsf{norm}(M)$ is still an RPM.
We call an RPM $M$ a \emph{normalized PRM} (\mbox{N-RPM}) if $\textsf{norm}(M) = M$.

Regarding N-RPMs in the graph $K^\bullet_n$, the following properties hold:
\begin{Property}\label{pro:norm-diff}
If two N-RPMs $M$ and $M'$ are different, $M'$ cannot be obtained from $M$ by using $\alpha$-rotate operators and reverse operators,
\end{Property}
\begin{Property}
A Kirkman matching $M^\mathrm{kir}_n$ is an N-RPM.
\end{Property}

\section{Rainbow perfect matchings in $K^\bullet_n$ with an even number of vertices}
Letting $k$ be a nonnegative integer, we show the results of considering the following two cases:
\begin{itemize}
    \item $n=8k$ or $n=8k+2$;
    \item $n=8k+4$ or $n=8k+6$.
\end{itemize}

\subsection{Non-existence of an RPM when $n=8k+4$ or $n=8k+6$}
We show the non-existence of RPM in the graph $K^\bullet_n$ for any $n\in\{8k+4,8k+6\}$.
\begin{Theorem}
   For any $k\in \mathbb{Z}_{\ge 0}$, no RPM exists in the graph $K^\bullet_n$ if $n\in\{8k+4,8k+6\}$.
\end{Theorem}
\begin{proof}
We color all vertices in $K^\bullet_n$ using a function $\chi$: $V(K^\bullet_n) \rightarrow \{black, white\}$:
\begin{align}\label{fun:v-color}
  \chi(v) =
  \begin{cases}
    black & \text{if the label of $v$ is an odd},\\
    white & \text{if the label of $v$ is an even}.
  \end{cases}
\end{align}

First consider the case where $n=8k+4$.
Function $\chi$ obtains $4k+2$ black vertices and $4k+2$ white ones in $K^\bullet_{8k+4}$.
Assume there exists an RPM $M$ in $K^\bullet_{8k+4}$; that is, $M$ is a matching in $K^\bullet_{8k+4}$ containing colors $c_1, c_2, \ldots, c_{4k+2}$.
According to the edge coloring in~\eqref{def:coloring}, edges with colors $c_1, c_3, \ldots, c_{4k+1}$ in $M$ occupy $2k+1$ black vertices and $2k+1$ white ones,
because the endpoints of each edge are colored in different colors.
The remaining edges in $M$ with colors $c_2, c_4, \ldots, c_{4k+2}$ consume an even number of vertices in both white and black,
because the endpoints of each edge are same-colored.
This contradicts the fact that numbers of white and black vertices in $K^\bullet_{8k+4}$ are both even.

A similar result holds for the case of $n=8k+6$, where using function $\chi$ in~(\ref{fun:v-color}) obtains $4k+3$ black vertices and $4k+3$ white ones.
Assume there exists an RPM $M$ in $K^\bullet_{8k+6}$ containing colors $c_1, c_2, \ldots, c_{4k+3}$.
The edges in $M$ with colors $c_1,c_3,\ldots,c_{4k+3}$ account for $2k+2$ black vertices and $2k+2$ white ones because the endpoints of each edge are differently colored.
Edges colored $c_2,c_4,\ldots,c_{4k+2}$ in $M$ require an even number of both black and white vertices,
because the endpoints of each edge are same-colored.
This contradicts the fact that the numbers of white and black vertices in $K^\bullet_{8k+6}$ are both odd.
\end{proof}

\subsection{Existence of an RPM when $n=8k$ or $n=8k+2$}
First consider the case where $k=0$.
The existence of an RPM is obvious because $\emptyset$ and $\{\{0,1\}\}$ are the RPMs in $K^\bullet_0$ and $K^\bullet_2$, respectively.
We then focus on $n=8k$ and $n=8k+2$ with $k\ge 1$.
For such cases, we design the following matching $T_n$ in the graph $K^\bullet_n$:
\begin{align}\label{def:t}
  T_n = T'_n \cup T''_n \cup T'''_n \cup \bar{T}_n
\end{align}
where 
\begin{align}
\label{def:t1}    &T'_n = 
    \begin{cases}
        \left\{\{1+i, 8k-2-i\} \mid i = 0,1,\dots,2k-3\right\} & \text{if $n=8k$},\\
        \left\{\{1+i, 8k-i\} \mid i = 0,1,\dots,2k-2\right\} & \text{if $n=8k+2$};
    \end{cases}\\
\label{def:t2}    &T''_n = 
    \begin{cases}
        \left\{\{2k+i,6k-i\} \mid i = 0,1,\dots,k-1\right\} & \text{if $n=8k$},\\
        \left\{\{2k+1+i,6k+1-i\} \mid i = 0,1,\dots,k-1\right\} & \text{if $n=8k+2$};
    \end{cases}\\
\label{def:t3}    &T'''_n = 
    \begin{cases}
        \left\{\{3k+i,5k-2-i\} \mid i = 0,1,\dots,k-2\right\} & \text{if $n=8k$},\\
        \left\{\{3k+1+i,5k-1-i\} \mid i = 0,1,\dots,k-2\right\} & \text{if $n=8k+2$};
    \end{cases}\\
\label{def:tb}    &\bar{T}_n =
    \begin{cases}
        \left\{\{0,4k-1\},\{2k-1,8k-1\},\{5k-1,5k\}\right\} & \text{if $n=8k$},\\
        \left\{\{0,2k\},\{4k,8k+1\},\{5k,5k+1\}\right\} & \text{if $n=8k+2$}.
    \end{cases}
\end{align}
\begin{Theorem}
   For any $n=8k$ or $8k+2$ with $k\in \mathbb{Z}_{\ge 1}$, $T_n$ is an RPM in the graph $K^\bullet_n$.
\end{Theorem}
\begin{proof}
   Tables~\ref{tab:t123-8k} and \ref{tab:t123-8k2} summarize the features of $T'_n$, $T''_n$, and $T'''_n$ from \eqref{def:t1}--\eqref{def:t3}.
    The columns ``vertices'' and ``colors'' show the covered vertices and colors, respectively.
    \begin{table}[ht]
        \centering
        \caption{Features of $T'_n$, $T''_n$, and $T'''_n$ when $n=8k$}
        \label{tab:t123-8k}
        \begin{tabular}{llllr}
        \hline
        && \multicolumn{1}{c}{vertices} & \multicolumn{1}{c}{colors} & \multicolumn{1}{c}{size} \\ \hline
        $T'_n$      && $\{1,2,\ldots,2k-2\} \cup \{6k+1,6k+2,\ldots,8k-2\}$& $\{c_3,c_5,\ldots,c_{4k-3}\}$ & $2k-2$ \\
        $T''_n$     && $\{2k,2k+1,\ldots,3k-1\} \cup \{5k+1,5k+2,\ldots,6k\}$ & $\{c_{2k+2},c_{2k+4},\ldots,c_{4k}\}$ & $k$ \\
        $T'''_n$    && $\{3k,3k+1,\ldots,4k-2\} \cup \{4k,4k+1,\ldots,5k-2\}$ & $\{c_2,c_4,\ldots,c_{2k-2}\}$ & $k-1$ \\ \hline
        \end{tabular}
    \end{table}
    \begin{table}[ht]
        \centering
        \caption{Features of $T'_n$, $T''_n$ and $T'''_n$ when $n=8k+2$}
        \label{tab:t123-8k2}
        \begin{tabular}{llllr}
        \hline
        && \multicolumn{1}{c}{vertices} & \multicolumn{1}{c}{colors} & \multicolumn{1}{c}{size} \\ \hline
        $T'_n$      && $\{1,2,\ldots,2k-1\} \cup \{6k+2,6k+3,\ldots,8k\}$& $\{c_3,c_5,\ldots,c_{4k-1}\}$ & $2k-1$ \\
        $T''_n$     && $\{2k+1,2k+2,\ldots,3k\} \cup \{5k+2,5k+3,\ldots,6k+1\}$ & $\{c_{2k+2},c_{2k+4},\ldots,c_{4k}\}$ & $k$ \\
        $T'''_n$    && $\{3k+1,3k+2,\ldots,4k-1\} \cup \{4k+1,4k+2,\ldots,5k-1\}$ & $\{c_2,c_4,\ldots,c_{2k-2}\}$ & $k-1$ \\ \hline
        \end{tabular}
    \end{table}
    
    For the case where $n=8k$, Table~\ref{tab:t123-8k} shows that $T'_n$, $T''_n$, and $T'''_n$ share neither vertices nor colors, and that vertices
    \begin{align*}
        0, 2k-1, 4k-1, 5k-1, 5k, 8k-1
    \end{align*}
    and colors
    \begin{align*}
        c_1, c_{2k}, c_{4k-1}
    \end{align*}
    remain unmatched.
    Edge set $\bar{T}_{8k} = \{\{0,4k-1\},\{2k-1,8k-1\},\{5k-1,5k\}\}$ satisfies all the remaining requirements, so $T_n$ is an RPM.
    
    For the case where $n=8k+2$, the same result that $T_n$ is an RPM can be obtained from Table~\ref{tab:t123-8k2} and the definition of $\bar{T}_n$.
\end{proof}

Figure~\ref{fig:t-1618} shows the examples of $T_{16}$ and $T_{18}$.
\input{T1618}

\section{Rainbow near-perfect matchings in $K^\bullet_n$ with an odd number of vertices}
Kirkman proposed Kirkman matching nearly 180 years ago. Property~\ref{pro:kick-RPM} indicates that an \mbox{N-RPM} exists in the graph $K^\bullet_n$ with an odd number of vertices.
However, even though multiple \mbox{N-RPMs} in the graph $K^\bullet_n$ with an odd number of vertices are required in real-world applications, it has not been confirmed whether \mbox{N-RPMs} other than Kirkman matching exist.

In this section, we first show the existence of an RPM through what we call \emph{arch-recursive-slide} (ARS) matching, whose \mbox{N-RPM} is different from the Kirkman matching when $n \in \{7,9,11,\ldots\}$.
We then propose an algorithm for generating multiple \mbox{N-RPMs} based on ARS matching.

\subsection{Arch-recursive-slide (ARS) matching}
In general, any odd number $n$ can be expressed as
\begin{align}
    & 8k+1, 8k+3, 8k+5, \mathrm{or\ } 8k+7 & k \in \mathbb{Z}_{\ge 0}.
\end{align} 
For the graph $K^\bullet_n$ with an odd number of vertices, we define the ARS matching $\Xi_n$ as 
\begin{align}\label{def:Xi}
      \Xi_n = \Xi'_n \cup \Xi''_n \cup \Xi'''_n
\end{align}
where 
\begin{align}\label{def:Xi1}
    \Xi'_n = 
    \begin{cases}
       \emptyset & \text{if $n=1$},\\
      \left\{\{i,2k-1-i\} \mid i = 0,1,\dots,k-1\right\} & \text{if $n=8k+1$ or $n=8k+3$}, \\ 
      \left\{\{i,2k+1-i\} \mid i = 0,1,\dots,k\right\} & \text{if $n=8k+5$ or $n=8k+7$};
    \end{cases}
\end{align}

\begin{align}\label{def:Xi2}
    \Xi''_n =
    \begin{cases}
       \emptyset & \text{if $n=1$},\\
        \left\{\{2k+i,4k+1+2i\} \mid i = 0,1,\dots,2k-1\right\} & \text{if $n=8k+1$},\\
        \left\{\{2k+i,4k+1+2i\} \mid i = 0,1,\dots,2k\right\} & \text{if $n = 8k+3$},\\
        \left\{\{2k+2+i,4k+4+2i\} \mid i = 0,1,\dots,2k\right\} & \text{if $n=8k+5$},\\
        \left\{\{2k+2+i,4k+4+2i\} \mid i = 0,1,\dots,2k+1\right\} & \text{if $n=8k+7$};
    \end{cases}
 \end{align}

 \begin{align}\label{def:Xi3}
    \Xi'''_n=
    \begin{cases}
       \emptyset & \text{if $n=1$},\\
       \left\{\{4k+2i,4k+2j\} \mid \{i,j\} \in \Xi_{2k+1} \right\} & \text{if $n=8k+1$},\\
       \left\{\{4k+2+2i,4k+2+2j\} \mid \{i,j\} \in \Xi_{2k+1} \right\} & \text{if $n=8k+3$},\\
       \left\{\{4k+3+2i,4k+3+2j\} \mid \{i,j\} \in \Xi_{2k+1}\right\} & \text{if $n=8k+5$},\\
       \left\{\{4k+5+2i,4k+5+2j\} \mid \{i,j\} \in \Xi_{2k+1} \right\} & \text{if $n=8k+7$}.\\
    \end{cases} 
 \end{align}

We show $\Xi_{33}$ as an example.
According to~\eqref{def:Xi1}--\eqref{def:Xi2}, we obtain
\begin{align}\label{val:Xi33-1}
   \Xi'_{33} = \left\{\{0,7\},\{1,6\},\{2,5\},\{3,4\}\right\},
\end{align}
\begin{align}\label{val:Xi33-2}
   \Xi''_{33} = \left\{\{8,17\},\{9,19\},\{10,21\},\{11,23\},\{12,25\},\{13,27\},\{14,29\},\{15,31\}\right\}.
\end{align}
From recursive formulation~\eqref{def:Xi3}, to obtain $\Xi'''_{33}$ we must compute $\Xi_{3}$ and $\Xi_{9}$ beforehand, as
\begin{align*}
\Xi_3   &= \Xi'_{3} \cup \Xi''_{3} \cup \Xi'''_{3}\\
        &= \emptyset \cup  \left\{\{0,1\}\right\} \cup \emptyset\\
        &= \left\{\{0,1\}\right\},\\
\Xi_9   &= \Xi'_{9} \cup \Xi''_{9} \cup \Xi'''_{9}\\
        &= \left\{\{0,1\}\right\} \cup \left\{\{2,5\},\{3,7\}\right\} \cup \left\{\{4+2i,4+2j\} \mid \{i,j\} \in \Xi_3 \right\}\\
        &= \left\{\{0,1\}\right\} \cup \left\{\{2,5\},\{3,7\}\right\} \cup  \left\{\{4,6\} \right\}\\
        &= \left\{\{0,1\},\{2,5\},\{3,7\},\{4,6\}\right\}
\end{align*}
Thus, by using~\eqref{def:Xi3}, we obtain $\Xi'''_{33}$ as
\begin{align}
\nonumber\Xi'''_{33}&= \left\{\{16+2i,16+2j\} \mid \{i,j\} \in \Xi_9 \right\}\\
\label{val:Xi33-3}  &= \left\{\{16,18\},\{20,26\},\{22,30\},\{24,28\}\right\}.
\end{align}
Finally, $\Xi_{33}$ is formed by~\eqref{val:Xi33-1}--\eqref{val:Xi33-3}:
\begin{align}
\nonumber   \Xi_{33}    &= \Xi'_{33} \cup \Xi''_{33} \cup \Xi'''_{33}\\
\nonumber               &= \left\{\{0,7\},\{1,6\},\{2,5\},\{3,4\},\{8,17\},\{9,19\},\{10,21\},\{11,23\},\{12,25\},\{13,27\},\{14,29\},\{15,31\}\right.\\
\label{val:Xi-33}                        &\hspace{4mm}\left.\{16,18\},\{20,26\},\{22,30\},\{24,28\}\right\}.
\end{align}
The ARS matching $\Xi_{33}$ is an RPM for the graph $K^\bullet_{33}$ because the edges in~\eqref{val:Xi-33} share no common vertices and cover all colors.
Figure~\ref{fig:Xi-33} shows $\Xi_{33}$ and its corresponding \mbox{N-RPM}.
\input{Ki33}

In this paper, we call an RPM $M$ in $K^\bullet_n$ a \emph{cuttable RPM} if 
\begin{align}\label{cdt:cuttable}
    \forall \{i,j\}\in M, |i-j| \le n-|i-j|
\end{align}
holds.
A necessary and sufficient condition for \eqref{cdt:cuttable} is 
\begin{align}
    \forall \{i,j\}\in M, |i-j|\le \frac{n}{2}.
\end{align}
Note that Kirkman matching $M^\mathrm{kir}_n$ is not cuttable where $n\ge 3$, but for any $n\in \mathbb{Z}_{\ge 0}$, matchings
\begin{align}\label{c-prm}
    \mathrm{rot}\left(M^\mathrm{kir}_n,\floor*{\frac{n+1}{4}}\right), \mathrm{rot}\left(M^\mathrm{kir}_n,-\floor*{\frac{n+1}{4}}\right)
\end{align} are cuttable RPMs.
Using this definition, we can prove that $\Xi_n$ is an RPM:
\begin{Theorem}\label{the:RPM}
For any odd number $n$, ARS matching $\Xi_n$ is an RPM in the graph $K^\bullet_n$.
\end{Theorem}
\begin{proof}
   The proof is by mathematical induction. For any $t\in \mathbb{Z}_{\ge 1}$, let $P(t)$ denote the statement that $\Xi_n$ is a cuttable RPM for each $n\in \{1,3,5,\ldots,8t-1\}$.
   \\
   {\bf Base step ($t=1$):} $P(1)$ is true because according to~\eqref{def:Xi1}--\eqref{def:Xi3}
   \begin{align*}
       \Xi_1&=\emptyset,\\
       \Xi_3&=\{\{0,1\}\},\\
       \Xi_5&=\{\{0,1\}, \{2,4\}\},\\
       \Xi_7&=\{\{0,1\}, \{2,4\},\{3,6\}\}.
   \end{align*}
   Each of these is an RPM in the corresponding graph $K^\bullet_n$, all involving edges $\{i, j\}$, satisfy $|i-j|\le n/2$.\\
   {\bf Inductive step $P(t)\rightarrow P(t+1)$:} Fix some $t\ge 1$, assuming that
   \begin{Assumption}\label{Ass:Pk}
   for any $n\in \{1,3,5,\ldots,8t-1\}$, $\Xi_n$ is a cuttable RPM.
   \end{Assumption}
   To prove $P(t+1)$, we must show that for any $n\in \{8t+1, 8t+3, 8t+5, 8t+7\}$, $\Xi_n$ is a cuttable RPM in the graph $K^\bullet_n$.
   Table~\ref{tab:Xi12} summarizes the features of $\Xi'_n$ and $\Xi''_n$ according to \eqref{def:Xi1}--\eqref{def:Xi2}, such as covered vertices (in the column ``vertices''), covered colors (in the column ``colors''), and size.
    \begin{table}[ht]
        \centering \fontsize{8pt}{9pt}\selectfont
        \setlength{\tabcolsep}{3.4pt}
        \caption{Features of $\Xi'_n$ and $\Xi''_n$}
        \label{tab:Xi12}
        \begin{tabular}{cccccc}
        \hline
                                && \multicolumn{4}{c}{$\Xi'_n$}\\ \cline{3-6}
        \multicolumn{1}{c}{$n$} && \multicolumn{1}{c}{vertices} & \multicolumn{1}{c}{colors} & $|\Xi'_n|$ & $\max_{\{i,j\}\in \Xi'_n} |i-j|$ \\ \hline
        $8k+1$                  && $\{0,1,\ldots,2k-1\}$ & $\{c_1,c_3,\ldots,c_{2k-1}\}$ & $k$ & $2k-1$\\
        $8k+3$                  && $\{0,1,\ldots,2k-1\}$ & $\{c_1,c_3,\ldots,c_{2k-1}\}$ & $k$ & $2k-1$\\
        $8k+5$                  && $\{0,1,\ldots,2k+1\}$ & $\{c_1,c_3,\ldots,c_{2k+1}\}$ & $k+1$ & $2k+1$\\
        $8k+7$                  && $\{0,1,\ldots,2k+1\}$ & $\{c_1,c_3,\ldots,c_{2k+1}\}$ & $k+1$ & $2k+1$ \\
        \hline
                                && \multicolumn{4}{c}{$\Xi''_n$}\\ \cline{3-6}
        \multicolumn{1}{c}{$n$} && \multicolumn{1}{c}{vertices} & \multicolumn{1}{c}{colors} & $|\Xi''_n|$ & $\max_{\{i,j\}\in \Xi''_n} |i-j|$ \\ \hline
        $8k+1$                  && $\{2k,2k+1,\ldots,4k-1\} \cup \{4k+1,4k+3,\ldots,8k-1\}$ & $\{c_{2k+1},c_{2k+2},\ldots,c_{4k}\}$ & $2k$ & $4k$\\
        $8k+3$                  && $\{2k,2k+1,\ldots,4k\} \cup \{4k+1,4k+3,\ldots,8k+1\}$ & $\{c_{2k+1},c_{2k+2},\ldots,c_{4k+1}\}$ & $2k+1$ & $4k+1$\\
        $8k+5$                  && $\{2k+2,2k+3,\ldots,4k+2\} \cup \{4k+4,4k+6,\ldots,8k+4\}$ & $\{c_{2k+2},c_{2k+3},\ldots,c_{4k+2}\}$ & $2k+1$ & $4k+2$\\
        $8k+7$                  && $\{2k+2,2k+3,\ldots,4k+3\} \cup \{4k+4,4k+6,\ldots,8k+6\}$ & $\{c_{2k+2},c_{2k+3},\ldots,c_{4k+3}\}$ & $2k+2$ & $4k+3$\\
        \hline
        \end{tabular}
    \end{table}
    Table~\ref{tab:Xi12} shows that $\Xi'_n$ and $\Xi''_n$ share neither common vertices nor colors.
    The values in the last column indicate that statement $|i-j| \le n/2$ holds for all edges $\{i,j\}\in \Xi'_n \cup \Xi''_n$.
    Table~\ref{tab:Xi3} lists the unmatched vertices and colors to be considered in $\Xi'''_{8t+1}, \Xi'''_{8t+3}, \Xi'''_{8t+5}$, and $\Xi'''_{8t+7}$.
    \begin{table}[ht]
        \centering
        \caption{Unmatched vertices and colors}
        \label{tab:Xi3}
        \begin{tabular}{llll}
        \hline
        \multicolumn{1}{c}{$n$} && \multicolumn{1}{c}{vertices} & \multicolumn{1}{c}{colors}\\ \hline
        $8t+1$                  && $\{4t,4t+2,\ldots,8t\}$ & $\{c_{2},c_{4},\ldots,c_{2t}\}$ \\
        $8t+3$                  && $\{4t+2,4t+4,\ldots,8t+2\}$ & $\{c_{2},c_{4},\ldots,c_{2t}\}$ \\
        $8t+5$                  && $\{4t+3,4t+5,\ldots,8t+3\}$ & $\{c_{2},c_{4},\ldots,c_{2t}\}$ \\
        $8t+7$                  && $\{4t+5,4t+7,\ldots,8t+5\}$ & $\{c_{2},c_{4},\ldots,c_{2t}\}$ \\ \hline
        \end{tabular}
    \end{table}
    From $t\ge 1$ and Assumption~\ref{Ass:Pk}, $\Xi_{2t+1}$ is a cuttable RPM for the graph $K^\bullet_{2t+1}$, using $2t$ distinct vertices in $\{0,1,\ldots,2t\}$.
    Thus, $\Xi'''_{8t+1}$, $\Xi'''_{8t+3}$, $\Xi'''_{8t+5}$, and $\Xi'''_{8t+7}$ constructed by~\eqref{def:Xi3} cover $2t$ distinct vertices in the column ``vertices'' in Table~\ref{tab:Xi3}.
    
    We next show that all colors in the column ``colors'' are covered.
    Assumption~\ref{Ass:Pk} guarantees that
    \begin{align}\label{eq:key-cover}
        \{ |i-j| \mid \{i,j\} \in \Xi_{2t+1} \} = \{ 1, 2, \ldots, t\},
    \end{align}
    so $\Xi'''_{8t+1}$, $\Xi'''_{8t+3}$, $\Xi'''_{8t+5}$, and $\Xi'''_{8t+7}$ constructed by~\eqref{def:Xi3} cover all colors in the column ``colors'' in Table~\ref{tab:Xi3}.
    Note that~\eqref{eq:key-cover} also indicates that
    \begin{align*}
        |i-j| \le 2t \le n/2
    \end{align*}
    holds for each $\{i,j\}$ in $\Xi'''_{8t+1}$, $\Xi'''_{8t+3}$, $\Xi'''_{8t+5}$, and $\Xi'''_{8t+7}$.

    Therefore, for each $n\in\{8t+1, 8t+3, 8t+5,8t+7\}$, $\Xi_n$ is an RPM in $K^\bullet_n$, and $|i-j|\le n/2$ holds for all edges $\{i,j\}$ in $\Xi_n$.
\end{proof}

For $n\in \{1,3,5\}$, all RPMs in the graph $K^\bullet_n$ are rooted from the same \mbox{N-RPM}.
However, if $n\in \{7,9,\ldots\}$, the \mbox{N-RPM} of ARS matching is different from Kirkman matching $M^\mathrm{kir}_n$ in the graph $K^\bullet_n$:
\begin{Property}
For any $n \in \{7,9,11,\ldots\}, \textsf{norm}(\Xi_n) \ne M^\mathrm{kir}_n$.
\end{Property}
\begin{proof}
From the definition of $\Xi''_n$ in~\eqref{def:Xi2}, $|\Xi''_n| \ge 2$ holds when $n \in \{7,9,\ldots\}$.
When $n = 8k+1$ or $n = 8k+3$, edges $\{2k,4k+1\}$ and $\{2k+1,4k+3\}$ belong to $\Xi''_n$.
If we arrange vertices $0, 1, \ldots, n-1$ around a cycle in clockwise order, these two edges cross each other,
and this cannot be changed by applying reverse and rotation operators.
However, all edges in Kirkman matching $M^\mathrm{kir}_n$ are parallel,
so the \mbox{N-RPM} of ARS matching is different from $M^\mathrm{kir}_n$ if $n = 8k+1$ or $n = 8k+3$.
For $n = 8k+5$ or $n = 8k+7$, the same holds for edges $\{2k+2, 4k+4\}$ and $\{2k+3, 4k+6\}$ in $\Xi''_n$.
\end{proof}

\subsection{Generating many \mbox{N-RPMs} based on ARS matching}
We showed in the previous subsection that the ARS matching $\Xi_n = \Xi'_n \cup \Xi''_n \cup \Xi'''_n$ is an RPM in $K^\bullet_n$ if $n$ is odd.
In this subsection, we propose a method for generating many \mbox{N-RPMs}, based on the idea that valid variants of $\Xi'''_{2k+1}$ lead to RPMs whose N-RPMs are different.

Given an RPM $M$ in $K^\bullet_n$ where $n$ is odd, we consider the following functions $f(M)$ and $g(M)$:
\begin{align}
    f(M) =
    \begin{cases}
    \mathrm{rot}(M,-2k)&\text{if $n=8k+1$ or $n=8k+3$},\\
    \mathrm{rot}(M,-2k-2)&\text{if $n=8k+5$ or $n=8k+7$};
    \end{cases}\\
    g(M) =
    \begin{cases}
    \mathrm{rot}(\mathrm{rev}(M),2k)&\text{if $n=8k+1$ or $n=8k+3$},\\
    \mathrm{rot}(\mathrm{rev}(M),2k+2)&\text{if $n=8k+5$ or $n=8k+7$}.
    \end{cases}
\end{align}
From Property~\ref{pro:rot}, Property~\ref{pro:rev}, and Theorem~\ref{the:RPM}, $f(\Xi_n)$, rev$(\Xi_n)$, and $g(\Xi_n)$ are RPMs in $K^\bullet_n$.
Note that the four RPMs $\Xi_n$, $f(\Xi_n)$, rev($\Xi_n$), and $g(\Xi_n)$ are all cuttable, and each is rooted from the same \mbox{N-RPM} by Property~\ref{pro:norm-diff}.

We demonstrate how to generate many \mbox{N-RPMs} by example.
In~\eqref{val:Xi-33}, we showed that $\Xi_{33}=\Xi'_{33} \cup \Xi''_{33} \cup \Xi'''_{33}$ is an RPM in $K^\bullet_{33}$.
We used $\Xi_9$ in constructing $\Xi'''_{33}$.
If we use $f(\Xi_9)$, rev$(\Xi_9)$ and $g(\Xi_9)$ instead of $\Xi_9$ when constructing $\Xi'''_{33}$, RPMs whose \mbox{N-RPMs} are different can be obtained for $K^\bullet_{33}$.
Note that other cuttable RPMs such as $\mathrm{rot}(M^\mathrm{kir}_9,2)$ and $\mathrm{rot}(M^\mathrm{kir}_9,-2)$ in~\eqref{c-prm} are also valid alternatives to $\Xi_9$.

Let $\Xi'''_n(M)$ denote 
\begin{align}
    \Xi'''_n(M)=
    \begin{cases}
       \left\{\{4k+2i,4k+2j\} \mid \{i,j\} \in M \right\} & \text{if $n=8k+1$},\\
       \left\{\{4k+2+2i,4k+2+2j\} \mid \{i,j\} \in M \right\} & \text{if $n=8k+3$},\\
       \left\{\{4k+3+2i,4k+3+2j\} \mid \{i,j\} \in M\right\} & \text{if $n=8k+5$},\\
       \left\{\{4k+5+2i,4k+5+2j\} \mid \{i,j\} \in M \right\} & \text{if $n=8k+7$}.\\
    \end{cases} 
\end{align}
We design a collection of RPMs $F_n$ in the graph $K^\bullet_n$ as
\begin{align}\label{def:Fn}
    &F_n = \\
    &\begin{cases}
       \left\{\emptyset \right\} & \text{if $n=1$},\\
        \{\Xi'_n \cup \Xi''_n \cup \Xi'''_n(M)
        \mid M \in F_{2k+1}\}\\
        \ \ \cup \{\Xi'_n \cup \Xi''_n \cup \Xi'''_n(f(M))
        \mid M \in F_{2k+1}\}\\
        \ \ \cup \{\Xi'_n \cup \Xi''_n \cup \Xi'''_n(\mathrm{rev}(M))
        \mid M \in F_{2k+1}\}\\
        \ \ \cup \{\Xi'_n \cup \Xi''_n \cup \Xi'''_n(g(M))
        \mid M \in F_{2k+1}\} & \text{if $n=8k+1,8k+3,8k+5,8k+7$}.
    \end{cases} 
\end{align}

Using this method, we obtain $\Theta(n)$ different \mbox{N-RPM}s in the graph $K^\bullet_n$, where $n$ is odd.

The following confirms the size of $F_n$ in detail.
For $n=1$, we obtain $|F_n|=1$ from~\eqref{def:Fn};
For $n=3,5,7$, we construct $F_n$ based on $F_1=\{\emptyset\}$.
Since $\emptyset = \mathrm{rev}(\emptyset) = f(\emptyset) = g(\emptyset)$,
$F_n=\{\Xi_n\}$ holds for $n=3,5,7$.
When $M=\Xi_3$ or $M=\Xi_5$, since $M = g(M)$ and $\mathrm{rev}(M) = f(M)$ hold
for $n=9,11,\ldots,23$, $F_n$ based on $\Xi_3, \Xi_5$ has 2 elements.
For the other $n=8k+i, i\in \{1,3,5,7\}$, $|F_n| = 4|F_{2k+1}|$ holds.

From the observations above, we can enumerate the size of $F_n$ as in Table~\ref{tab:Fn},
where values in the column ``$|F_n|$'' were confirmed by computational experiments.
    \begin{table}[ht]
        \centering
        \caption{The size of $F_n$ confirmed by computational experiments}
        \label{tab:Fn}
        \begin{tabular}{llll}
        \hline
        \multicolumn{1}{c}{$n$} && \multicolumn{1}{c}{$2k+1$} & \multicolumn{1}{c}{$|F_n|$}\\ \hline
        $1,3,5,7$                  && $1$ & 1 \\
        $9,11,\ldots,23$                  && $3,5$ & 2 \\
        $25,27,\ldots,31$                  && $7$ & 4 \\
        $33,35,\ldots,95$                  && $9,11,\ldots,23$ & 8 \\
        $97,99,\ldots,127$                  && $25,27,\ldots,31$ & 16 \\
        $129,131,\ldots,383$                  && $33,35,\ldots,95$ & 32 \\
        $385,387,\ldots,511$                  && $97,99,\ldots,127$ & 64 \\
        $513,515,\ldots,1535$                  && $129,131,\ldots,383$ & 128 \\
        $1537,1539,\ldots,2047$                  && $385,387,\ldots,511$ & 256 \\
        $2049,2051,\ldots,6143$                  && $513,515,\ldots,1535$ & 512 \\ \hline
        \end{tabular}
    \end{table}

\section*{Acknowledgement}
This work was supported by JSPS KAKENHI [Grant No.19K11843].

\bibliographystyle{unsrt}
\bibliography{references}

\end{document}

%% file: K7-K8.tex
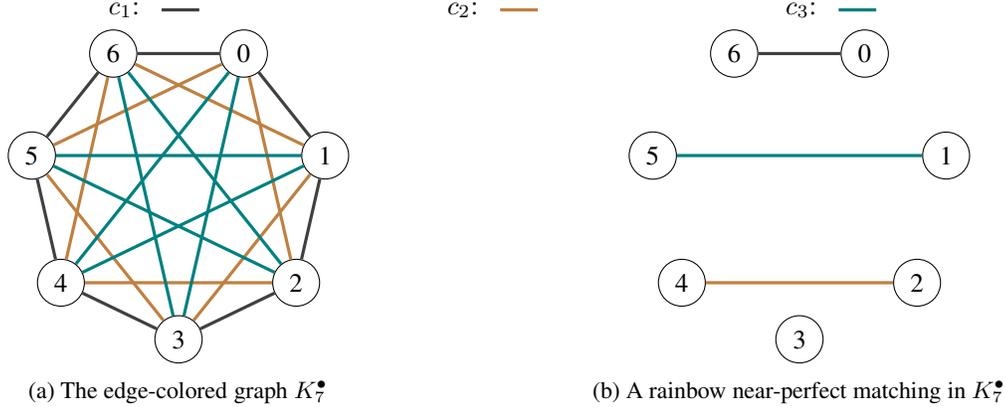
\begin{figure}[ht]
\centering
\begin{tikzpicture}
    \node at (-5, 0) {$c_1$: };
    \draw[darkgray, very thick] (-4.5, 0) -- (-4, 0);
    \node at (-0.5, 0) {$c_2$: };
    \draw[brown, very thick] (0, 0) -- (0.5, 0);
    \node at (4, 0) {$c_3$: };
    \draw[teal, very thick] (4.5, 0) -- (5, 0);
\end{tikzpicture}\\
\begin{subfigure}{.5\textwidth}
    \centering
    \begin{tikzpicture}
        \edef\N{7}
        \pgfmathparse{int(\N-1)}
        \edef\NM{\pgfmathresult}
        \foreach \n in {0,...,\NM}{
            \node[circle,draw,minimum size = .1cm](\n) at (90-180/\N-360/\N*\n:2){\n}; 
        }
        
        \draw (0) edge[darkgray, very thick] (1);
        \draw (1) edge[darkgray, very thick] (2);
        \draw (2) edge[darkgray, very thick] (3);
        \draw (3) edge[darkgray, very thick] (4);
        \draw (4) edge[darkgray, very thick] (5);
        \draw (5) edge[darkgray, very thick] (6);
        \draw (6) edge[darkgray, very thick] (0);
        
        \draw (0) edge[brown, very thick] (2);
        \draw (1) edge[brown, very thick] (3);
        \draw (2) edge[brown, very thick] (4);
        \draw (3) edge[brown, very thick] (5);
        \draw (4) edge[brown, very thick] (6);
        \draw (5) edge[brown, very thick] (0);
        \draw (6) edge[brown, very thick] (1);
        
        \draw (0) edge[teal, very thick] (3);
        \draw (1) edge[teal, very thick] (4);
        \draw (2) edge[teal, very thick] (5);
        \draw (3) edge[teal, very thick] (6);
        \draw (4) edge[teal, very thick] (0);
        \draw (5) edge[teal, very thick] (1);
        \draw (6) edge[teal, very thick] (2);
    \end{tikzpicture}
    \caption{The edge-colored graph $K^\bullet_7$}
\end{subfigure}%
\begin{subfigure}{.5\textwidth}
    \centering
    \begin{tikzpicture}
        \edef\N{7}
        \pgfmathparse{int(\N-1)}
        \edef\NM{\pgfmathresult}
        \foreach \n in {0,...,\NM}{
            \node[circle,draw,minimum size = .1cm](\n) at (90-180/\N-360/\N*\n:2){\n}; 
        }
        \draw (0) edge[darkgray, very thick] (6);
        \draw (2) edge[brown, very thick] (4);
        \draw (1) edge[teal, very thick] (5);
    \end{tikzpicture}
    \caption{A rainbow near-perfect matching in $K^\bullet_7$} \label{fig:rpm-k*7}
\end{subfigure}
\caption{Examples of $K^\bullet_7$} \label{fig:k*7}
\end{figure}

\begin{figure}[ht]
\centering
\begin{tikzpicture}
    \node at (-5, 0) {$c_1$: };
    \draw[darkgray, very thick] (-4.5, 0) -- (-4, 0);
    \node at (-2, 0) {$c_2$: };
    \draw[brown, very thick] (-1.5, 0) -- (-1, 0);
    \node at (1, 0) {$c_3$: };
    \draw[teal, very thick] (1.5, 0) -- (2, 0);
    \node at (4, 0) {$c_4$: };
    \draw[purple, very thick] (4.5, 0) -- (5, 0);
\end{tikzpicture}\\
\begin{subfigure}{.5\textwidth}
    \centering
    \begin{tikzpicture}
        \edef\N{8}
        \pgfmathparse{int(\N-1)}
        \edef\NM{\pgfmathresult}
        \foreach \n in {0,...,\NM}{
            \node[circle,draw,minimum size = .1cm](\n) at (90-180/\N-360/\N*\n:2){\n}; 
        }
        \draw (0) edge[darkgray, very thick] (1);
        \draw (1) edge[darkgray, very thick] (2);
        \draw (2) edge[darkgray, very thick] (3);
        \draw (3) edge[darkgray, very thick] (4);
        \draw (4) edge[darkgray, very thick] (5);
        \draw (5) edge[darkgray, very thick] (6);
        \draw (6) edge[darkgray, very thick] (7);
        \draw (7) edge[darkgray, very thick] (0);
        
        \draw (0) edge[brown, very thick] (2);
        \draw (1) edge[brown, very thick] (3);
        \draw (2) edge[brown, very thick] (4);
        \draw (3) edge[brown, very thick] (5);
        \draw (4) edge[brown, very thick] (6);
        \draw (5) edge[brown, very thick] (7);
        \draw (6) edge[brown, very thick] (0);
        \draw (7) edge[brown, very thick] (1);
        
        \draw (0) edge[teal, very thick] (3);
        \draw (1) edge[teal, very thick] (4);
        \draw (2) edge[teal, very thick] (5);
        \draw (3) edge[teal, very thick] (6);
        \draw (4) edge[teal, very thick] (7);
        \draw (5) edge[teal, very thick] (0);
        \draw (6) edge[teal, very thick] (1);
        \draw (7) edge[teal, very thick] (2);
        
        \draw (0) edge[purple, very thick] (4);
        \draw (1) edge[purple, very thick] (5);
        \draw (2) edge[purple, very thick] (6);
        \draw (3) edge[purple, very thick] (7);
        \draw (4) edge[purple, very thick] (0);
        \draw (5) edge[purple, very thick] (1);
        \draw (6) edge[purple, very thick] (2);
        \draw (7) edge[purple, very thick] (3);
    \end{tikzpicture}
    \caption{The edge-colored graph $K^\bullet_8$}
\end{subfigure}%
\begin{subfigure}{.5\textwidth}
    \centering
    \begin{tikzpicture}
        \edef\N{8}
        \pgfmathparse{int(\N-1)}
        \edef\NM{\pgfmathresult}
        \foreach \n in {0,...,\NM}{
            \node[circle,draw,minimum size = .1cm](\n) at (90-180/\N-360/\N*\n:2){\n}; 
        }
        \draw (0) edge[darkgray, very thick] (1);
        \draw (4) edge[brown, very thick] (6);
        \draw (2) edge[teal, very thick] (5);
        \draw (7) edge[purple, very thick] (3);
    \end{tikzpicture}
    \caption{A rainbow perfect matching in $K^\bullet_8$} \label{fig:rpm-k*8}
\end{subfigure}
\caption{Examples of $K^\bullet_8$} \label{fig:k*8}
\end{figure}
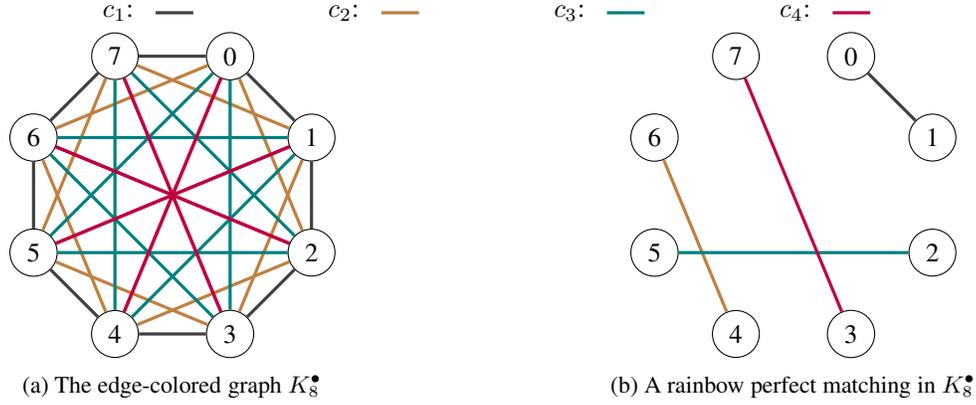

%% file: RotateK7.tex
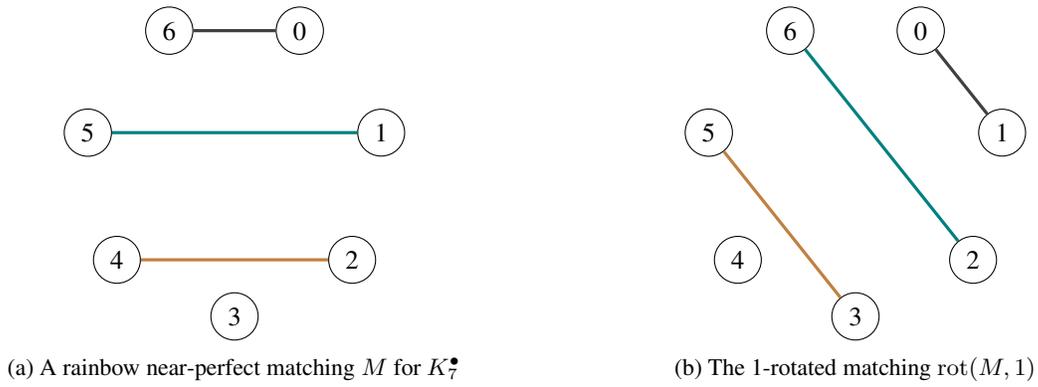
\begin{figure}[ht]
\centering
\begin{subfigure}{.5\textwidth}
    \centering
    \begin{tikzpicture}
        \edef\N{7}
        \pgfmathparse{int(\N-1)}
        \edef\NM{\pgfmathresult}
        \foreach \n in {0,...,\NM}{
            \node[circle,draw,minimum size = .1cm](\n) at (90-180/\N-360/\N*\n:2){\n}; 
        }
        \draw (0) edge[darkgray, very thick] (6);
        \draw (2) edge[brown, very thick] (4);
        \draw (1) edge[teal, very thick] (5);
    \end{tikzpicture}
    \caption{A rainbow near-perfect matching $M$ for $K^\bullet_7$}
\end{subfigure}%
\begin{subfigure}{.5\textwidth}
    \centering
    \begin{tikzpicture}
        \edef\N{7}
        \pgfmathparse{int(\N-1)}
        \edef\NM{\pgfmathresult}
        \foreach \n in {0,...,\NM}{
            \node[circle,draw,minimum size = .1cm](\n) at (90-180/\N-360/\N*\n:2){\n}; 
        }
        \draw (1) edge[darkgray, very thick] (0);
        \draw (3) edge[brown, very thick] (5);
        \draw (2) edge[teal, very thick] (6);
    \end{tikzpicture}
    \caption{The \mbox{1-rotated} matching $\mathrm{rot}(M,1)$} \label{fig:rrpm-k*7}
\end{subfigure}
\caption{Example 1-rotated RPM} \label{fig:rpm-rk7}
\end{figure}

%% file: ReverseK8.tex
\begin{figure}[ht]
\centering
\begin{subfigure}{.5\textwidth}
    \centering
    \begin{tikzpicture}
        \edef\N{8}
        \pgfmathparse{int(\N-1)}
        \edef\NM{\pgfmathresult}
        \foreach \n in {0,...,\NM}{
            \node[circle,draw,minimum size = .1cm](\n) at (90-180/\N-360/\N*\n:2){\n}; 
        }
        \draw (0) edge[darkgray, very thick] (1);
        \draw (4) edge[brown, very thick] (6);
        \draw (2) edge[teal, very thick] (5);
        \draw (7) edge[purple, very thick] (3);
    \end{tikzpicture}
    \caption{A rainbow perfect matching $M$ in $K^\bullet_8$}
\end{subfigure}%
\begin{subfigure}{.5\textwidth}
    \centering
    \begin{tikzpicture}
        \edef\N{8}
        \pgfmathparse{int(\N-1)}
        \edef\NM{\pgfmathresult}
        \foreach \n in {0,...,\NM}{
            \node[circle,draw,minimum size = .1cm](\n) at (90-180/\N-360/\N*\n:2){\n}; 
        }
        \draw (7) edge[darkgray, very thick] (6);
        \draw (3) edge[brown, very thick] (1);
        \draw (5) edge[teal, very thick] (2);
        \draw (0) edge[purple, very thick] (4);
    \end{tikzpicture}
    \caption{The reversed matching $\mathrm{rev}(M)$}
\end{subfigure}
\caption{Example reversed RPM} \label{fig:rpm-rk8}
\end{figure}
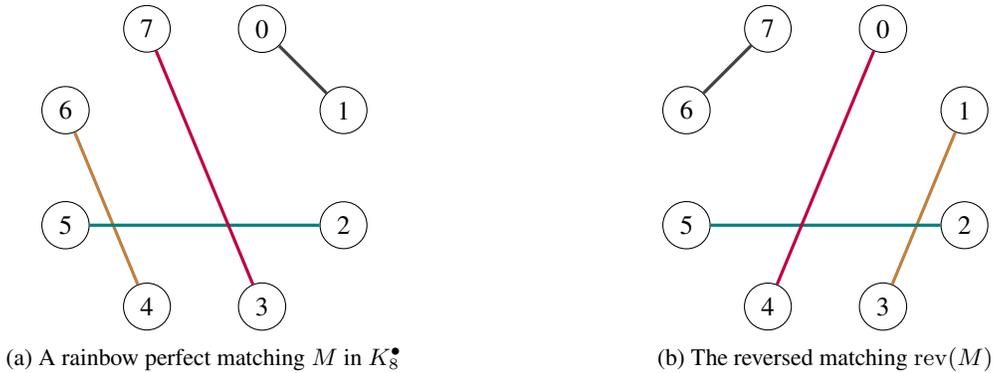

%% file: RTSP8.tex
\begin{figure}[ht]
\edef\N{7}
\pgfmathparse{int(\N-1)}
\edef\NM{\pgfmathresult}
\centering
\begin{subfigure}{.5\textwidth}
    \centering
    \begin{tikzpicture}
        \foreach \n in {0,...,\NM}{
            \node[circle,draw,minimum size = .1cm](\n) at (90-180/\N-360/\N*\n:1.8){\n}; 
        }
        \node[circle,draw,minimum size = .1cm] at (0, 0) (7) {7};
        \draw (0) edge[darkgray, very thick] (6);
        \draw (1) edge[brown, very thick] (5);
        \draw (2) edge[teal, very thick] (4);
        \draw (3) edge[loosely dashed, very thick] (7);
    \end{tikzpicture}
    \caption{Games in round 1}
\end{subfigure}%
\begin{subfigure}{.5\textwidth}
    \centering
    \begin{tikzpicture}
        \foreach \n in {0,...,\NM}{
            \node[circle,draw,minimum size = .1cm](\n) at (90-180/\N-360/\N*\n:1.8){\n}; 
        }
        \node[circle,draw,minimum size = .1cm] at (0, 0) (7) {7};
        \draw (1) edge[darkgray, very thick] (0);
        \draw (2) edge[brown, very thick] (6);
        \draw (3) edge[teal, very thick] (5);
        \draw (4) edge[loosely dashed, very thick] (7);
    \end{tikzpicture}
    \caption{Games in round 2}
\end{subfigure}\\
\vspace{5mm}
\begin{subfigure}{.5\textwidth}
    \centering
    \begin{tikzpicture}
        \foreach \n in {0,...,\NM}{
            \node[circle,draw,minimum size = .1cm](\n) at (90-180/\N-360/\N*\n:1.8){\n}; 
        }
        \node[circle,draw,minimum size = .1cm] at (0, 0) (7) {7};
        \draw (2) edge[darkgray, very thick] (1);
        \draw (3) edge[brown, very thick] (0);
        \draw (4) edge[teal, very thick] (6);
        \draw (5) edge[loosely dashed, very thick] (7);
    \end{tikzpicture}
    \caption{Games in round 3}
\end{subfigure}%
\begin{subfigure}{.5\textwidth}
    \centering
    \begin{tikzpicture}
        \foreach \n in {0,...,\NM}{
            \node[circle,draw,minimum size = .1cm](\n) at (90-180/\N-360/\N*\n:1.8){\n}; 
        }
        \node[circle,draw,minimum size = .1cm] at (0, 0) (7) {7};
        \draw (3) edge[darkgray, very thick] (2);
        \draw (4) edge[brown, very thick] (1);
        \draw (5) edge[teal, very thick] (0);
        \draw (6) edge[loosely dashed, very thick] (7);
    \end{tikzpicture}
    \caption{Games in round 4}
\end{subfigure}\\
\vspace{5mm}
\begin{subfigure}{.5\textwidth}
    \centering
    \begin{tikzpicture}
        \foreach \n in {0,...,\NM}{
            \node[circle,draw,minimum size = .1cm](\n) at (90-180/\N-360/\N*\n:1.8){\n}; 
        }
        \node[circle,draw,minimum size = .1cm] at (0, 0) (7) {7};
        \draw (4) edge[darkgray, very thick] (3);
        \draw (5) edge[brown, very thick] (2);
        \draw (6) edge[teal, very thick] (1);
        \draw (0) edge[loosely dashed, very thick] (7);
    \end{tikzpicture}
    \caption{Games in round 5}
\end{subfigure}%
\begin{subfigure}{.5\textwidth}
    \centering
    \begin{tikzpicture}
        \foreach \n in {0,...,\NM}{
            \node[circle,draw,minimum size = .1cm](\n) at (90-180/\N-360/\N*\n:1.8){\n}; 
        }
        \node[circle,draw,minimum size = .1cm] at (0, 0) (7) {7};
        \draw (5) edge[darkgray, very thick] (4);
        \draw (6) edge[brown, very thick] (3);
        \draw (0) edge[teal, very thick] (2);
        \draw (1) edge[loosely dashed, very thick] (7);
    \end{tikzpicture}
    \caption{Games in round 6}
\end{subfigure}\\
\vspace{5mm}
\begin{subfigure}{.5\textwidth}
    \centering
    \begin{tikzpicture}
        \foreach \n in {0,...,\NM}{
            \node[circle,draw,minimum size = .1cm](\n) at (90-180/\N-360/\N*\n:1.8){\n}; 
        }
        \node[circle,draw,minimum size = .1cm] at (0, 0) (7) {7};
        \draw (6) edge[darkgray, very thick] (5);
        \draw (0) edge[brown, very thick] (4);
        \draw (1) edge[teal, very thick] (3);
        \draw (2) edge[loosely dashed, very thick] (7);
    \end{tikzpicture}
    \caption{Games in round 7}
\end{subfigure}
\caption{Feasible schedule of 8 teams based on a perfect rainbow matching of $K^\bullet_7$} \label{fig:rtsp8}
\end{figure}
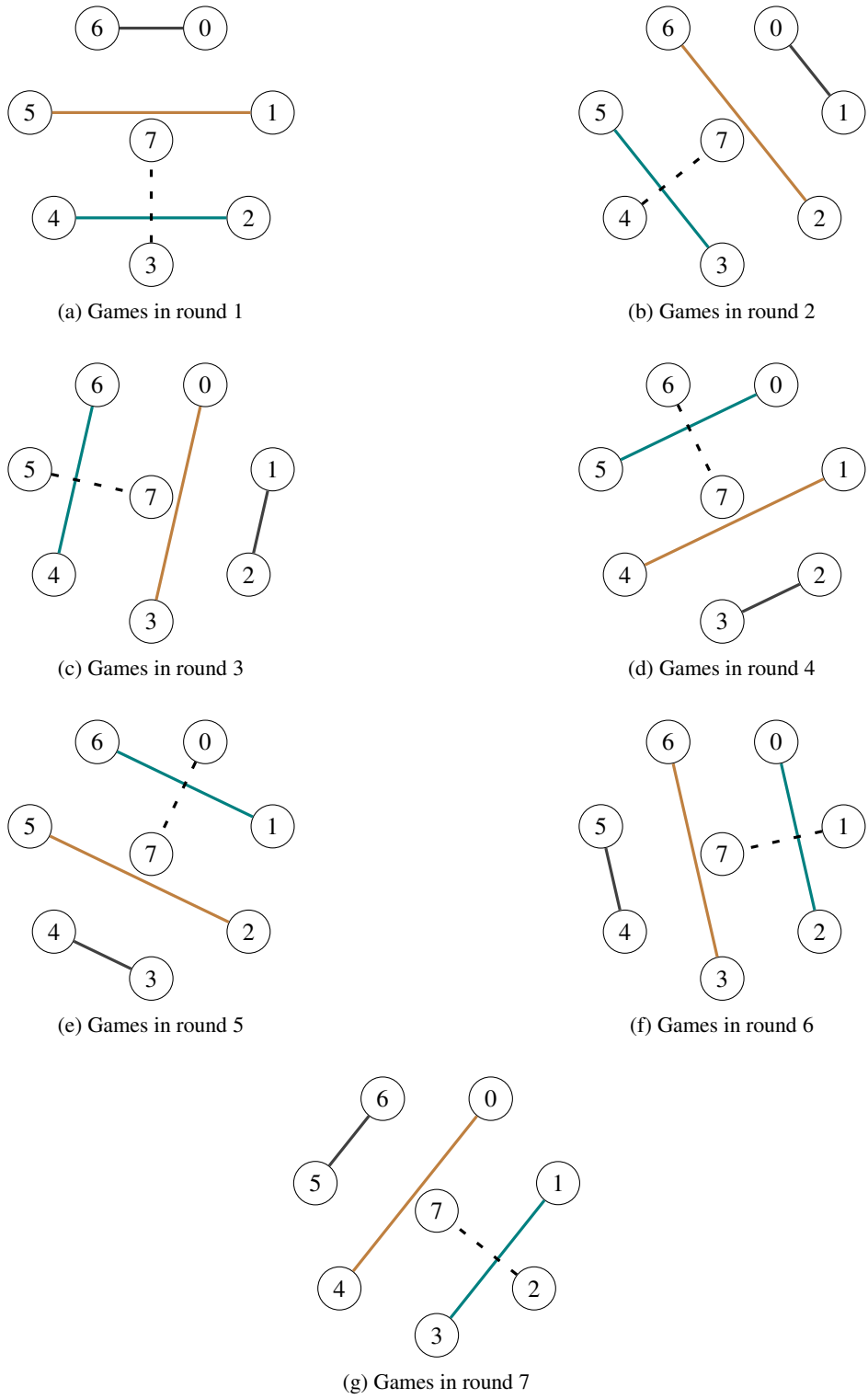

%% file: T1618.tex
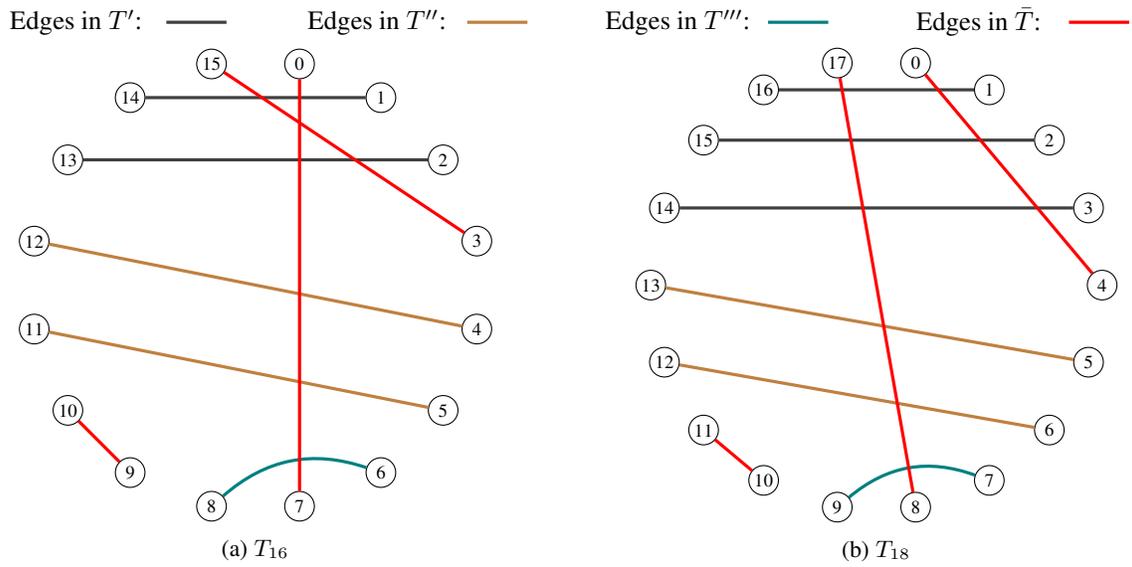
\begin{figure}[ht]
\centering
\begin{tikzpicture}
    \node at (-7, 0) {Edges in $T'$: };
    \draw[darkgray, very thick] (-5.8, 0) -- (-5, 0);
    \node at (-3, 0) {Edges in $T''$: };
    \draw[brown, very thick] (-1.8, 0) -- (-1, 0);
    \node at (1, 0) {Edges in $T'''$: };
    \draw[teal, very thick] (2.2, 0) -- (3, 0);
    \node at (5, 0) {Edges in $\bar{T}$: };
    \draw[red, very thick] (6.2, 0) -- (7, 0);
\end{tikzpicture}
\begin{subfigure}{.5\textwidth}
    \centering
    \begin{tikzpicture}
        \edef\N{16}
        \pgfmathparse{int(\N-1)}
        \edef\NM{\pgfmathresult}
        \pgfmathparse{int(\N/8)}
        \edef\k{\pgfmathresult}
        \foreach \n in {0,...,\NM}{
            \node[circle,draw,inner sep=0pt,scale=0.7,minimum size=16pt](\n) at (90-180/\N-360/\N*\n:3){\n}; 
        }
        \pgfmathparse{int(2*\k-3)}
        \edef\IEND{\pgfmathresult}
        \foreach \i in {0,...,\IEND}{
            \pgfmathparse{int(1+\i)}
            \edef\s{\pgfmathresult}
            \pgfmathparse{int(8*\k-2-\i)}
            \edef\t{\pgfmathresult}
            \draw (\s) edge[darkgray, very thick] (\t);
        }
        \pgfmathparse{int(\k-1)}
        \edef\IEND{\pgfmathresult}
        \foreach \i in {0,...,\IEND}{
            \pgfmathparse{int(2*\k+\i)}
            \edef\s{\pgfmathresult}
            \pgfmathparse{int(6*\k-\i)}
            \edef\t{\pgfmathresult}
            \draw (\s) edge[brown, very thick] (\t);
        }
        \pgfmathparse{int(\k-2)}
        \edef\IEND{\pgfmathresult}
        \foreach \i in {0,...,\IEND}{
            \pgfmathparse{int(3*\k+\i)}
            \edef\s{\pgfmathresult}
            \pgfmathparse{int(5*\k-2-\i)}
            \edef\t{\pgfmathresult}
            \draw (\s) edge[teal, very thick, bend right] (\t);
        }
        \draw (0) edge[red, very thick] (7);
        \draw (3) edge[red, very thick] (15);
        \draw (9) edge[red, very thick] (10);
    \end{tikzpicture}
    \caption{$T_{16}$}
\end{subfigure}%
\begin{subfigure}{.5\textwidth}
    \centering
    \begin{tikzpicture}
        \edef\N{18}
        \pgfmathparse{int(\N-1)}
        \edef\NM{\pgfmathresult}
        \pgfmathparse{int(\N/8)}
        \edef\k{\pgfmathresult}
        \foreach \n in {0,...,\NM}{
            \node[circle,draw,inner sep=0pt,scale=0.7,minimum size=16pt](\n) at (90-180/\N-360/\N*\n:3){\n}; 
        }
        \pgfmathparse{int(2*\k-2)}
        \edef\IEND{\pgfmathresult}
        \foreach \i in {0,...,\IEND}{
            \pgfmathparse{int(1+\i)}
            \edef\s{\pgfmathresult}
            \pgfmathparse{int(8*\k-\i)}
            \edef\t{\pgfmathresult}
            \draw (\s) edge[darkgray, very thick] (\t);
        }
        \pgfmathparse{int(\k-1)}
        \edef\IEND{\pgfmathresult}
        \foreach \i in {0,...,\IEND}{
            \pgfmathparse{int(2*\k+1+\i)}
            \edef\s{\pgfmathresult}
            \pgfmathparse{int(6*\k+1-\i)}
            \edef\t{\pgfmathresult}
            \draw (\s) edge[brown, very thick] (\t);
        }
        \pgfmathparse{int(\k-2)}
        \edef\IEND{\pgfmathresult}
        \foreach \i in {0,...,\IEND}{
            \pgfmathparse{int(3*\k+1+\i)}
            \edef\s{\pgfmathresult}
            \pgfmathparse{int(5*\k-1-\i)}
            \edef\t{\pgfmathresult}
            \draw (\s) edge[teal, very thick, bend right] (\t);
        }
        \draw (0) edge[red, very thick] (4);
        \draw (8) edge[red, very thick] (17);
        \draw (10) edge[red, very thick] (11);
    \end{tikzpicture}
    \caption{$T_{18}$}
\end{subfigure}
\caption{Examples of RPM matchings $T_{16}$ and $T_{18}$} \label{fig:t-1618}
\end{figure}

%% file: Ki33.tex
\begin{figure}[ht]
\centering
\begin{tikzpicture}
    \node at (-8, 0) {Edges in $\Xi'_{33}$: };
    \draw[darkgray, very thick] (-6.8, 0) -- (-6, 0);
    \node at (-3, 0) {Edges in $\Xi''_{33}$: };
    \draw[brown, very thick] (-1.8, 0) -- (-1, 0);
    \node at (2, 0) {Edges in $\Xi'''_{33}$: };
    \draw[teal, very thick] (3.2, 0) -- (4, 0);
\end{tikzpicture}
\begin{subfigure}{.5\textwidth}
    \centering
    \begin{tikzpicture}
        \edef\N{33}
        \pgfmathparse{int(\N-1)}
        \edef\NM{\pgfmathresult}
        \pgfmathparse{int(\N/8)}
        \edef\k{\pgfmathresult}
        \foreach \n in {0,...,\NM}{
            \node[circle,draw,inner sep=0pt,scale=0.7,minimum size=16pt,fill=white](\n) at (90-180/\N-360/\N*\n:3){\n}; 
        }
        \pgfmathparse{int(\k-1)}
        \edef\IEND{\pgfmathresult}
        \foreach \i in {0,...,\IEND}{
            \pgfmathparse{int(2*\k-1-\i)}
            \edef\j{\pgfmathresult}
            \draw (\i) edge[darkgray, very thick, bend right] (\j);
        }
        \pgfmathparse{int(2*\k-1)}
        \edef\LEND{\pgfmathresult}
        \foreach \l in {0,...,\LEND}{
            \pgfmathparse{int(2*\k+\l)}
            \edef\i{\pgfmathresult}
            \pgfmathparse{int(4*\k+1+2*\l)}
            \edef\j{\pgfmathresult}
            \draw (\i) edge[brown, very thick] (\j);
        }
        \draw (16) edge[teal, very thick, bend right] (18);
        \draw (20) edge[teal, very thick, bend right] (26);
        \draw (22) edge[teal, very thick, bend right] (30);
        \draw (24) edge[teal, very thick, bend right] (28);
    \end{tikzpicture}
    \caption{ARS matching $\Xi_{33}$}
\end{subfigure}%
\begin{subfigure}{.5\textwidth}
    \centering
    \begin{tikzpicture}
        \edef\N{33}
        \pgfmathparse{int(\N-1)}
        \edef\NM{\pgfmathresult}
        \pgfmathparse{int(\N/8)}
        \edef\k{\pgfmathresult}
        \foreach \n in {0,...,\NM}{
            \node[circle,draw,inner sep=0pt,scale=0.7,minimum size=16pt,fill=white](\n) at (90-180/\N-360/\N*\n:3){\n}; 
        }
        \pgfmathparse{int(\k-1)}
        \edef\LEND{\pgfmathresult}
        \foreach \i in {0,...,\LEND}{
            \pgfmathparse{int(\N-1-\i)}
            \edef\j{\pgfmathresult}
            \draw (\i) edge[darkgray, very thick, bend left] (\j);
        }
        \pgfmathparse{int(2*\k-1)}
        \edef\LEND{\pgfmathresult}
        \foreach \l in {0,...,\LEND}{
            \pgfmathparse{int(\k+\l)}
            \edef\i{\pgfmathresult}
            \pgfmathparse{int(3*\k+1+2*\l)}
            \edef\j{\pgfmathresult}
            \draw (\i) edge[brown, very thick] (\j);
        }
        \draw (12) edge[teal, very thick, bend right] (14);
        \draw (16) edge[teal, very thick, bend right] (22);
        \draw (18) edge[teal, very thick, bend right] (26);
        \draw (20) edge[teal, very thick, bend right] (24);
    \end{tikzpicture}
    \caption{The \mbox{N-RPM} of $\Xi_{33}$}
\end{subfigure}
\caption{ARS matching $\Xi_{33}$ and its \mbox{N-RPM}} \label{fig:Xi-33}
\end{figure}
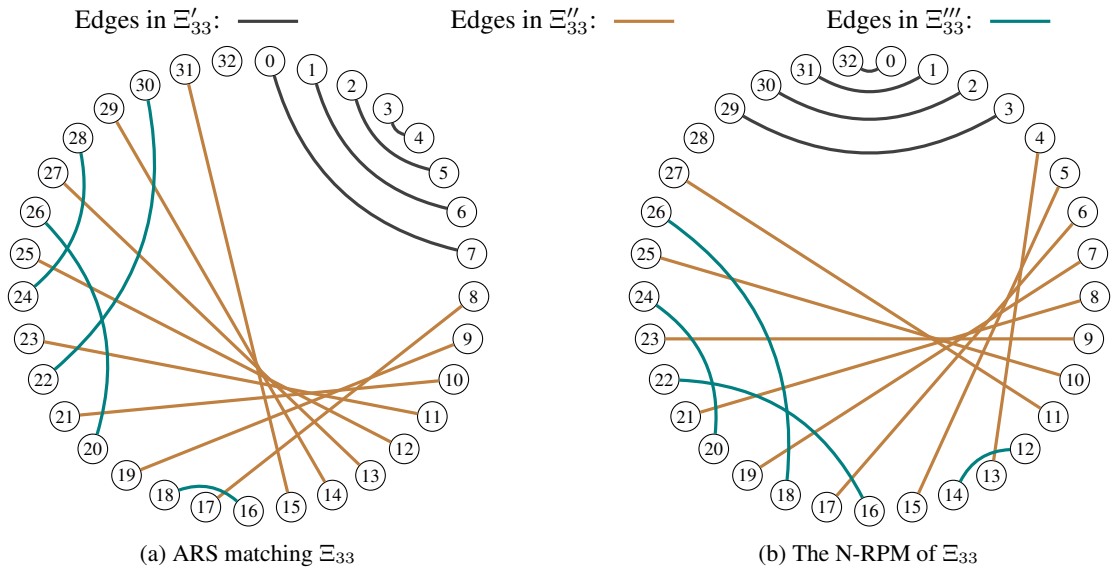

%% file: ms.bbl
\begin{thebibliography}{10}

\bibitem{garey1979computers}
Michael~R Garey and David~S Johnson.
\newblock {\em Computers and intractability}, volume 174.
\newblock freeman San Francisco, 1979.

\bibitem{perarnau2013rainbow}
Guillem Perarnau and Oriol Serra.
\newblock Rainbow perfect matchings in complete bipartite graphs: existence and
  counting.
\newblock {\em Combinatorics, Probability and Computing}, 22(5):783--799, 2013.

\bibitem{cano2016rainbow}
Mar{\'\i}a del~Pilar Cano~Vila, Guillem Perarnau, and Oriol Serra~Alb{\'o}.
\newblock Rainbow perfect matchings in r-partite graph structures.
\newblock {\em Electronic notes in discrete mathematics}, 54:193--198, 2016.

\bibitem{coulson2019rainbow}
Matthew Coulson and Guillem Perarnau.
\newblock Rainbow matchings in dirac bipartite graphs.
\newblock {\em Random Structures \& Algorithms}, 55(2):271--289, 2019.

\bibitem{bal2017rainbow}
Deepak Bal, Patrick Bennett, Xavier P{\'e}rez-Gim{\'e}nez, and Pawe{\l}
  Pra{\l}at.
\newblock Rainbow perfect matchings and hamilton cycles in the random geometric
  graph.
\newblock {\em Random Structures \& Algorithms}, 51(4):587--606, 2017.

\bibitem{bal2016rainbow}
Deepak Bal and Alan Frieze.
\newblock Rainbow matchings and hamilton cycles in random graphs.
\newblock {\em Random Structures \& Algorithms}, 48(3):503--523, 2016.

\bibitem{wang2008heterochromatic}
Guanghui Wang and Hao Li.
\newblock Heterochromatic matchings in edge-colored graphs.
\newblock {\em the electronic journal of combinatorics}, pages R138--R138,
  2008.

\bibitem{lesaulnier2010rainbow}
Timothy~D LeSaulnier, Christopher Stocker, Paul~S Wenger, and Douglas~B West.
\newblock Rainbow matching in edge-colored graphs.
\newblock {\em The Electronic Journal of Combinatorics [electronic only]},
  17(1):v17i1n26--pdf, 2010.

\bibitem{kostochka2012large}
Alexandr Kostochka and Matthew Yancey.
\newblock Large rainbow matchings in edge-coloured graphs.
\newblock {\em Combinatorics, Probability and Computing}, 21(1-2):255--263,
  2012.

\bibitem{lo2015existences}
Allan Lo.
\newblock Existences of rainbow matchings and rainbow matching covers.
\newblock {\em Discrete Mathematics}, 338(11):2119--2124, 2015.

\bibitem{wang2011rainbow}
Guanghui Wang.
\newblock Rainbow matchings in properly edge colored graphs.
\newblock {\em the electronic journal of combinatorics}, pages P162--P162,
  2011.

\bibitem{diemunsch2011rainbow}
Jennifer Diemunsch, Michael Ferrara, Casey Moffatt, Florian Pfender, and Paul~S
  Wenger.
\newblock Rainbow matchings of size$\backslash$delta (g) in properly
  edge-colored graphs.
\newblock {\em arXiv preprint arXiv:1108.2521}, 2011.

\bibitem{gyarfas2014rainbow}
Andr{\'a}s Gy{\'a}rf{\'a}s and G{\'a}bor~N S{\'a}rk{\"o}zy.
\newblock Rainbow matchings and cycle-free partial transversals of latin
  squares.
\newblock {\em Discrete Mathematics}, 327:96--102, 2014.

\bibitem{aharoni2019large}
Ron Aharoni, Eli Berger, Maria Chudnovsky, David Howard, and Paul Seymour.
\newblock Large rainbow matchings in general graphs.
\newblock {\em European Journal of Combinatorics}, 79:222--227, 2019.

\bibitem{babu2015rainbow}
Jasine Babu, L~Sunil Chandran, and Krishna Vaidyanathan.
\newblock Rainbow matchings in strongly edge-colored graphs.
\newblock {\em Discrete Mathematics}, 338(7):1191--1196, 2015.

\bibitem{wang2016existence}
Guanghui Wang, Guiying Yan, and Xiaowei Yu.
\newblock Existence of rainbow matchings in strongly edge-colored graphs.
\newblock {\em Discrete Mathematics}, 339(10):2457--2460, 2016.

\bibitem{cheng2018note}
Yangyang Cheng, Ta~Sheng Tan, and Guanghui Wang.
\newblock A note on rainbow matchings in strongly edge-colored graphs.
\newblock {\em Discrete Mathematics}, 341(10):2762--2767, 2018.

\bibitem{kano2008monochromatic}
Mikio Kano and Xueliang Li.
\newblock Monochromatic and heterochromatic subgraphs in edge-colored graphs-a
  survey.
\newblock {\em Graphs and Combinatorics}, 24(4):237--263, 2008.

\bibitem{fujita2009rainbow}
Shinya Fujita, Atsushi Kaneko, Ingo Schiermeyer, and Kazuhiro Suzuki.
\newblock A rainbow $ k $-matching in the complete graph with $ r $ colors.
\newblock {\em the electronic journal of combinatorics}, pages R51--R51, 2009.

\bibitem{haas2012anti}
Ruth Haas and Michael Young.
\newblock The anti-ramsey number of perfect matching.
\newblock {\em Discrete Mathematics}, 312(5):933--937, 2012.

\bibitem{fujita2010rainbow}
Shinya Fujita, Colton Magnant, and Kenta Ozeki.
\newblock Rainbow generalizations of ramsey theory: a survey.
\newblock {\em Graphs and Combinatorics}, 26(1):1--30, 2010.

\bibitem{de1980geography}
Dominique De~Werra.
\newblock Geography, games and graphs.
\newblock {\em Discrete Applied Mathematics}, 2(4):327--337, 1980.

\bibitem{de1981scheduling}
Dominique De~Werra.
\newblock Scheduling in sports.
\newblock {\em Studies on Graphs and Discrete Programming}, 11:381--395, 1981.

\bibitem{januario2016edge}
Tiago Januario, Sebasti{\'a}n Urrutia, Celso~C Ribeiro, and Dominique De~Werra.
\newblock Edge coloring: A natural model for sports scheduling.
\newblock {\em European Journal of Operational Research}, 254(1):1--8, 2016.

\bibitem{Kirkman1847on}
T.P. Kirkman.
\newblock On a problem in combinations.
\newblock {\em The Cambridge and Dublin Mathematical Journal}, 2:191--204,
  1847.

\bibitem{goossens2012soccer}
Dries~R Goossens and Frits~CR Spieksma.
\newblock Soccer schedules in europe: an overview.
\newblock {\em Journal of scheduling}, 15(5):641--651, 2012.

\bibitem{miyashiro2006minimizing}
Ryuhei Miyashiro and Tomomi Matsui.
\newblock Minimizing the carry-over effects value in a round-robin tournament.
\newblock In {\em Proceedings of the 6th International Conference on the
  Practice and Theory of Automated Timetabling}, pages 460--463. PATAT, 2006.

\bibitem{lambrechts2018round}
Erik Lambrechts, Annette~MC Ficker, Dries~R Goossens, and Frits~CR Spieksma.
\newblock Round-robin tournaments generated by the circle method have maximum
  carry-over.
\newblock {\em Mathematical Programming}, 172(1-2):277--302, 2018.

\end{thebibliography}
